\documentclass[11pt]{article}

\usepackage{amsfonts,amssymb,amsmath}

\newcommand\Sym{S}
\newcommand\hlg{{\Lie(\hG)}}
\renewcommand\lg{{\Lie(G)}}
\newcommand\inv{^{-1}}

\newcommand\Proj{{\rm Proj}}

\newcommand\face{{\mathcal F}}
\newcommand\mcL{{\mathcal L}}

\makeatletter
\newcommand\revddots{\mathinner{\mkern1mu\raise\p@\vbox{\kern7\p@\hbox{.}}\mkern2mu\raise4\p@\hbox{.}\mkern2mu\raise7\p@\hbox{.}\mkern1mu}}
\makeatother

\newenvironment{proof}{{\noindent\bf Proof.}}{\hfill $\square$}
\newenvironment{NB}{{\noindent\bf Remark.}}{}

\newcommand\LR{\operatorname{LR}}
\newcommand\mode{\operatorname{mod}}
\newcommand\lr{{\QQ_{\geq 0}\LR}}
\newcommand\longto{\longrightarrow}

\newcommand\Part{{\mathcal S}}
\newcommand\Kron{\operatorname{Kron}}

\newcommand\PP{{\mathbb P}}
\newcommand\QQ{{\mathbb Q}}\newcommand\ZZ{{\mathbb
    Z}}\newcommand\NN{{{\mathbb Z}_{\geq 0}}}
\newcommand\CC{{\mathbb C}}

\newcommand\Fl{{\mathcal Fl}}\newcommand\Face{{\mathcal F}}
\newcommand\Gr{{\mathbb G}}

\newcommand\Pic{\operatorname{Pic}}

\newcommand\SL{\operatorname{SL}}

\newcommand\Ho{\operatorname{H}}

\newcommand\Li{{\mathcal{L}}}
\newcommand\Mi{{\mathcal{M}}}

\newcommand\quot{/\hspace{-.5ex}/}

\newcommand\Lie{{\operatorname{Lie}}}
\newcommand\GL{\operatorname{GL}}
\newcommand\Wt{\operatorname{Wt}}

\newcommand\hnu{{\hat\nu}}
\newcommand\hG{{\hat G}} \newcommand\hB{{\hat B}}\newcommand\hT{{\hat T}} 
\newcommand\hW{{\hat W}}\newcommand\hw{{\hat w}}
\newcommand\hP{{\hat P}}

\newtheorem{lemma}{Lemma}
\newtheorem{prop}{Proposition}
\newtheorem{theo}{Theorem}

\renewcommand\binom[1]{
 \left (\begin{array}{@{}c@{}}
 #1\\
 2
 \end{array}	
 \right )}
	
\begin{document}
\title{Horn inequalities for nonzero Kronecker coefficients}
\author{N. Ressayre}


\maketitle
\begin{abstract}
The Kronecker coefficients $g_{\alpha\beta\gamma}$ and the Littlewood-Richardson coefficients $c_{\alpha\beta}^\gamma$ are nonnegative integers depending on three partitions $\alpha$, $\beta$, and $\gamma$. By definition,  $g_{\alpha\beta\gamma}$  (resp. $c_{\alpha\beta}^\gamma$) are the  multiplicities of the tensor product decomposition of two irreducible representations of symmetric groups (resp. linear groups). By a classical Littlewood-Murnaghan's result the Kronecker coefficients extend the Littlewood-Richardson ones.

The nonvanishing of the Littlewood-Richardson coefficient
$c_{\alpha\beta}^\gamma$ implies that $(\alpha, \beta, \gamma)$
satisfies some linear inequalities called Horn inequalities. In this
paper, we extend the essential Horn inequalities to 
the triples of partitions corresponding to a nonzero Kronecker
coefficient.

Along the way, we describe the set of tripless $(\alpha,\beta,\gamma)$
of partitions such that $c_{\alpha\beta}^\gamma\neq 0$ and
$l(\alpha)\leq e$, $l(\beta)\leq f$ and $l(\gamma)\leq e+f$, for some
given positive integers $e$ and $f$. 
This set is the natural analogue of the classical Horn semigroup when
one thinks about $c_{\alpha\beta}^\gamma$ as the branching
multiplicities for the subgroup $\GL_e\times\GL_f$ of $\GL_{e+f}$.
\end{abstract}

\section{Introduction}

If $\alpha=(\alpha_1\geq\alpha_2\geq\cdots\geq\alpha_e\geq 0)$ is a partition, we
set $|\alpha|=\sum_i\alpha_i$ in such a way $\alpha$ is a partition
of $|\alpha|$.
Consider the symmetric group $\Sym_n$  on $n$ letters.
The irreducible representations of $\Sym_n$ are parametrized by the
partitions of $n$, see {\it e.g.} \cite[I. 7]{Macdo} . Let $[\alpha]$ denote the representation
of $\Sym_{|\alpha|}$ corresponding to $\alpha$.
The Kronecker coefficients $g_{\alpha\,\beta\,\gamma}$, depending on
three partitions $\alpha,\,\beta$, and $\gamma$ of the same integer $n$, are defined by

\begin{eqnarray}
  \label{eq:defk}  [\alpha]\otimes [\beta]=\sum_{\gamma}g_{\alpha\,\beta\,\gamma}[\gamma].
\end{eqnarray}

The length $l(\alpha)$ of the partition $\alpha$ is the number of
nonzero parts $\alpha_i$.
Let $V$ be  a complex vector space of dimension $d$.
If  $l(\alpha)\leq d$ then $S^\alpha V$ denotes 
the Schur power (see {\it e.g.} \cite{FH}): it is an
irreducible polynomial representation of the linear group $\GL(V)$.
Let  $\beta$ be a second  partition such that $l(\beta)\leq d$.
 Then the Littlewood-Richardson coefficients
$c_{\alpha\,\beta}^\gamma$ are defined by 
\begin{eqnarray}
  \label{eq:defLR}  S^\alpha V\otimes S^\beta
  V=\sum_{\gamma}c_{\alpha\,\beta}^\gamma S^\gamma V.
\end{eqnarray}

The partition obtained by suppressing the first part of $\alpha$ is
denoted by $\bar\alpha=(\alpha_2\geq\alpha_3\dots)$.
Observe  that $\bar\alpha_1=\alpha_2$.
We state a classical result due to Littlewood and Murnaghan
(see for example \cite{JaKe}).

\begin{prop}\label{prop:Murna}
Let $\alpha$, $\beta$ and $\gamma$ be three partitons of the same
integer $n$.
\begin{enumerate}
\item  If $g_{\alpha\,\beta\,\gamma}\neq 0$ then 
    \begin{eqnarray}
      \label{ineq:Murn}
      (n-\alpha_1)+(n-\beta_1)\geq n-\gamma_1.
    \end{eqnarray}
\item If $(n-\alpha_1)+(n-\beta_1)= n-\gamma_1$ then 
\begin{eqnarray}
      \label{eq:redMurn}
g_{\alpha\,\beta\,\gamma}=c_{\bar\alpha\,\bar\beta}^{\bar\gamma}.
  \end{eqnarray}

\end{enumerate}
 
\end{prop}

In this paper, we prove many other inequalities similar to the identity~\eqref{ineq:Murn}, that are consequences of the nonvanishing of $g_{\alpha\,\beta\,\gamma}$.
For the partitions $(\alpha,\beta,\gamma)$ satisfying  equality in such an
inequality, we prove a reduction rule for $g_{\alpha\,\beta\,\gamma}$
similar to the identity~\eqref{eq:redMurn}. 

Observe that the formula~\eqref{eq:redMurn} shows that the Kronecker
coefficients extend the Littlewood-Richardson ones. 
Indeed, given $\bar \alpha$, $\bar\beta$ and $\bar \gamma$, one can
find $\alpha=(\alpha_1,\bar
\alpha),\beta=(\beta_1,\bar \beta)$ and $\gamma=(\gamma_1,\bar\gamma)$
such that $|\alpha|=|\beta|=|\gamma|=:n$,
$(n-\alpha_1)+(n-\beta_1)=n-\gamma_1$. Then 
$ c_{\bar\alpha\,\bar\beta}^{\bar\gamma}=g_{\alpha\,\beta\,\gamma}$ is
a Kronecker coefficient.
If $ c_{\bar\alpha\,\bar\beta}^{\bar\gamma}\neq 0$ then 
 $(\bar\alpha,\bar\beta,\bar\gamma)$ satisfy the  Horn inequalities (see {\it e.g.} \cite{Fulton:survey} or below for details). 
If  $g_{\alpha\,\beta\,\gamma}\neq 0$, our inequalities for $(\alpha,\beta,\,\gamma)$ 
extend some Horn inequalities.
Fix such an inequality $\varphi(\bar \alpha,\bar\beta,\bar \gamma)\geq
0$. We want to find an inequality
$\tilde\varphi(\alpha,\beta,\gamma)\geq 0$ such that
\begin{enumerate}
\item If $g_{\alpha\,\beta\,\gamma}\neq 0$ then
  $\tilde\varphi(\alpha,\beta,\gamma)\geq 0$;
\item If $(n-\alpha_1)+(n-\beta_1)=n-\gamma_1$ then $\tilde\varphi(\alpha,\beta,\gamma)=\varphi(\bar \alpha,\bar\beta,\bar \gamma)$.
\end{enumerate}
For example, a Weyl's  theorem \cite{Weyl:ineq} asserts that if  $ c_{\bar\alpha\,\bar\beta}^{\bar\gamma}\neq 0$ then
\begin{eqnarray}
\label{eq:Weyl1}
\bar\gamma_{e+j-1}\leq \bar\beta_{j-1},
\end{eqnarray}
whenever $l(\bar\alpha)\leq e$ and $j\geq 2$.

Before stating  our  extension of Weyl's  theorem, we introduce some notation. 
Let $\Part(r,d)$ denote the set of subsets of  $\{1,\cdots,d\} $
 with $r$ elements.
 Given $I=\{i_1<\cdots<i_r\}\in\Part(r,d)$ and $\alpha=(\alpha_1\geq\cdots\geq\alpha_d)$ a partition of length at most $d$, we set $\alpha_I=(\alpha_{i_1}\geq\cdots\geq\alpha_{i_r})$.
 Observe that $\bar\alpha_I=(\alpha_{i_1+1}\geq\cdots\geq\alpha_{i_r+1})$.

\begin{theo}\label{th:WeylKron}
Let $e$ and $f$ be two positive integers. 
Let  $\alpha$, $\beta$, and $\gamma$ be three partitions of the same
integer $n$ such that 
\begin{eqnarray}
  \label{eq:condl}
  l(\alpha)\leq e+1,\ \ 
 l(\beta)\leq f+1,\ \mbox{ and }\ 
l(\gamma)\leq e+f+1.
\end{eqnarray}
Let $j\in\{2,\dots,f+1\}$.

\begin{enumerate}
\item If
$g_{\alpha\,\beta\,\gamma}\neq 0$
then
\begin{eqnarray}
\label{ineq:WeylKron}
n+\gamma_1+\gamma_{e+j}\leq \alpha_1+\beta_1+\beta_j
\end{eqnarray}
\item Set $J=\{1,\dots,f\}-\{j-1\}$ and $K=\{1,\dots,e+f\}-\{e+j-1\}$.
If $n+\gamma_1+\gamma_{e+j}= \alpha_1+\beta_1+\beta_j$ then
$$
g(\alpha,\beta,\gamma)= \sum_{l(x)\leq 2e,\,l(y)\leq 2}
c(x,\bar\beta_J;\bar\gamma_K)\cdot
c(\gamma_1\geq\gamma_j,y;\beta_1\geq\beta_j)
\cdot 
g(\bar\alpha,x,y).
$$
\end{enumerate}
\end{theo}

\begin{NB}
  In the statement of Theorem~\ref{th:WeylKron} (and sometimes below)
  we denote $c_{\alpha\,\beta}^\gamma$ and $g_{\alpha\,\beta\,\gamma}$
  respectively by $c(\alpha,\beta;\gamma)$ and $g(\alpha,\beta,\gamma)$.
\end{NB}

\bigskip
Theorem~\ref{th:WeylKron} extends Weyl's theorem in the sense that if 
$(n-\alpha_1)+(n-\beta_1)= n-\gamma_1$ then the inequality~\eqref{ineq:WeylKron} is equivalent
to $\gamma_{e+j}\leq\beta_j$, that is to the inequality~\eqref{eq:Weyl1}.\\

For $I\in \Part(r,d)$,
consider the partition
$$
\tau^I=(d-r+1-i_1\geq d-r+2-i_2\geq\cdots\geq d-i_r).
$$
Set $|\alpha_I|:=\sum_{i\in
  I}\alpha_i$.
Observe that 
$|\bar\alpha_I|:=\sum_{i\in
  I}\alpha_{i+1}$.
We can now state our main result.

\begin{theo}
  \label{th:HornKron}
Let  $\alpha$, $\beta$, and $\gamma$ be three partitions of the same
integer $n$ 
satisfying the conditions~\eqref{eq:condl}.

Assume that  $g_{\alpha\,\beta\,\gamma}\neq 0$. 
Then
\begin{eqnarray}
  \label{ineq:HornKron}
n+|\bar\alpha_I|-\alpha_1+|\bar\beta_J|-\beta_1\geq 
|\bar\gamma_K|- \gamma_1,
\end{eqnarray}
for any $0<r<e$, $0<s<f$,
$I\in\Part(r,e)$, $J\in\Part(s,f)$ and $K\in\Part(r+s,e+f)$ such
that
\begin{eqnarray}
  \label{eq:condcohom}
  c_{\tau^I\,\tau^J}^{\tau^K}=1.
\end{eqnarray}
\end{theo}

If $(n-\alpha_1)+(n-\beta_1)= n-\gamma_1$ then the inequality~\eqref{ineq:HornKron} is equivalent
to 
\begin{equation}
  \label{eq:v1} 
|\bar\alpha_I|+|\bar\beta_J|\geq 
|\bar\gamma_K|,
 \end{equation}
which is a Horn inequality (see \cite{Fulton:survey} or Section~\ref{sec:Horn}).\\

\begin{NB}
Since inequalities~\eqref{ineq:Murn}, \eqref{ineq:WeylKron} and
\eqref{ineq:HornKron} are linear in $(\alpha,\beta,\gamma)$, the
condition $g_{\alpha\,\beta\,\gamma}\neq 0$ in
Proposition~\ref{prop:Murna} and  Theorem~\ref{th:WeylKron} and
\ref{th:HornKron} can be replaced by the weaker condition $g_{k\alpha\,k\beta\,k\gamma}\neq 0$ 
for some positive $k$.
\end{NB}

We get a reduction formula  for the coefficients $g_{\alpha\,\beta\,\gamma}$  if 
 the inequality~\eqref{ineq:HornKron} is saturated. 
 If $I\in\Part(r,d)$, we 
  denote by  $I_-\in \Part(d-r,d)$ the complement of $I$ in $\{1,\dots,d\}$.
 By symmetry we also set $I_+=I$.
 
 \begin{theo}
 \label{th:reduction}
Let  $\alpha$, $\beta$, and $\gamma$ be three partitions of the same
integer $n$ 
satisfying the conditions~\eqref{eq:condl}.

 Let $(I,J,K)$ be a triple that appears in Theorem~\ref{th:HornKron}
 (in particular satisfying the condition~\eqref{eq:condcohom}). 
 We assume   that
 \begin{eqnarray}
  \label{eq:HornKron}
n+|\bar\alpha_I|-\alpha_1+|\bar\beta_J|-\beta_1=
|\bar\gamma_K|- \gamma_1. 
\end{eqnarray}
Then $g(\alpha,\beta,\gamma)$ is equal to
\begin{equation}
 \label{eq:red}
\begin{array}{ll}
\sum_{a,b,x,y,u,v} &
c(\bar\alpha_{I_-}, \bar\beta_{J_-};y)\cdot
c(x,y;\bar\gamma_{K_+})\cdot
c(u,v;\bar\gamma_{K_-})\cdot\\
&c(a,u;\bar\alpha_{I}) \cdot
c(b,v;\bar\beta_{J}) \cdot
g(a,b,x),
\end{array} 
\end{equation}

where the sum runs  over the partitions $a,b,x,y,u,v$ satisfying
\begin{eqnarray}
\label{eq:indicesum}
\begin{array}{lll}
l(x)\leq (e-r)(f-s),&l(a)\leq e-r,&l(u)\leq e-r,\\[-3pt]
l(y)\leq r+s,&l(b)\leq f-s,&l(v)\leq f-s.
\end{array}
\end{eqnarray}
 \end{theo}
 
Note that in Theorem~\ref{th:reduction}, we needn't assume that
$g_{\alpha\beta\gamma}\neq 0$.

 Let $\Kron(e+1,f+1,e+f+1)$ denote the set of triples  $(\alpha,\,\beta,\,\gamma)$ of partitions 
  such that $|\alpha|=|\beta|=|\gamma|$, $g_{\alpha\beta\gamma}\neq 0$ and 
$l(\alpha)\leq e+1,\,l(\beta)\leq f+1,\,l(\gamma)\leq e+f+1$.
Then $\Kron(e+1,f+1,e+f+1)$ is a finitely generated semigroup in
$\NN^{2e+2f+3}$. 
In particular, the  cone 
$\QQ_{\geq 0}\Kron(e+1,f+1,e+f+1)$ generated by $\Kron(e+1,f+1,e+f+1)$ is a 
closed convex polyhedral cone. 

\begin{theo}
\label{th:codimoneface}
The inequalities~\eqref{ineq:WeylKron} in Theorem~\ref{th:WeylKron}
and the inequalities~\eqref{ineq:HornKron} in Theorem~\ref{th:HornKron} are essential, that is correspond 
to codimension one faces of $\QQ_{\geq 0}\Kron(e+1,f+1,e+f+1)$.
\end{theo}

One can guess to describe the complete minimal list $\mcL$ of
inequalities characterizing $\QQ_{\geq 0}\Kron(e+1,f+1,e+f+1)$.
Such a list is known for the Littlewood-Richardson coefficients (see
Theorem~\ref{th:LR} below for details). In principle, \cite{GITEigen}
gives $\mcL$. 
Nevertheless, it is known to be untractable to make this description
very explicit. Indeed, one first need to describe the so-called
adapted one-parameter subgroups by describing the collection of
hyperplanes spanned by subsets of a given set: a tricky combinatorial
problem. And secondly one need to understand an unknown Schubert problem. 
In this paper we describe a natural subset of $\mcL$ related with the
Horn cone. 
   
\bigskip
Inequality~\eqref{ineq:Murn} defines a codimension one face $\face_{LM}$
of $\QQ_{\geq 0}\Kron(e+1,f+1,e+f+1)$. Here ``$LM$'' stands for
Littlewood-Murnaghan. Each Horn inequality~\eqref{eq:v1} or Weyl
inequality~\eqref{eq:Weyl1} define a face $\face$ of codimension two
contained in  $\face_{LM}$. By convex geometry $\face$ has to be
contained in a second codimension one face $\face'$ of $\QQ_{\geq
  0}\Kron(e+1,f+1,e+f+1)$. 
Basically, Theorem~\ref{th:WeylKron} and  \ref{th:HornKron} describe
this face $\face'$. 

\bigskip
\noindent{\bf Comparaison between Theorems~\ref{th:WeylKron} and \ref{th:HornKron}.}
With  $I,J$ and $K$ respectively equal to $\{1,\dots,e\}$,
$\{1,\dots,f\}-\{j-1\}$,
and $\{1,\dots,e+f\}-\{e+j-1\}$
(where $j\in\{2,\dots,f+1\}$), we have 
$c_{\tau^I\,\tau^J}^{\tau^K}=1$.
The inequality~\eqref{ineq:HornKron} gives 
\begin{eqnarray}\label{eq:ineqred}
2n+2\gamma_1+\gamma_j\geq 2\alpha_1+2\beta_1+\beta_j.
\end{eqnarray}
This inequality is satisfied if $g_{\alpha\,\beta\,\gamma}\neq 0$. 
But the corresponding face has codimension 2 in $\QQ_{\geq
  0}\Kron(e+1,f+1,e+f+1)$.
Hence the inequality~\eqref{eq:ineqred} is not essential.
More precisely, it is a consequence
of  inequalities~\eqref{ineq:Murn} and \eqref{ineq:WeylKron}.

\bigskip
In Section~\ref{sec:sgrp}, we define and compare several semigroups. 
In Section \ref{sec:GITEigen}, we recall some results from
  \cite{GITEigen} that allows to describe some cones generated by
  these semigroups. In Section~\ref{sec:Horn}, we describe the support
  of the LR-coefficients $c_{\alpha\beta}^\gamma$ for partitions
  satisfying $l(\alpha)\leq e$, $l(\beta)\leq f$ and $l(\gamma)\leq
  e+f$, for fixed positive integers $e$ and $f$. Note that these
  assumptions are natural if one thinks about the LR-coefficients as
  multiplicities for the branching from $\GL_e\times\GL_f$ to
  $\GL_{e+f}$.
It is a variation of the classical Horn problem. 
The next sections contain the proofs of the statements  of the
introduction. 

\bigskip
{\bf Acknowledgements.} The author is partially supported by the French National Agency
(Project GeoLie ANR-15-CE40-0012) and the Institut Universitaire de
France (IUF).

\section{Semigroups}
\label{sec:sgrp}

 \subsection{Definitions}

\subsubsection{Kronecker semigroups}
We extend the definition of $g_{\alpha\beta\gamma}$ to any triple
$(\alpha,\beta,\gamma)$ of partitions by setting $g_{\alpha\beta\gamma}=0$
if the condition $|\alpha|=|\beta|=|\gamma|$ does not hold.
Let $e$, $f$, and $g$ be three positive integers. We define
  $\Kron(e,f,g)$ to be the set of triples $(\alpha,\,\beta,\,\gamma)$ of partitions 
  such that $g_{\alpha\beta\gamma}\neq 0$ and 
$l(\alpha)\leq e,\,l(\beta)\leq f,\,l(\gamma)\leq g$.
It is well known that $\Kron(e,f,g)$ is a finitely generated semigroup
of $\NN^{e+f+g}$.\\

\subsubsection{ Littlewood-Richardson semigroups}
We define 
  $\LR(e,f,g)$ to be the set of triples $(\alpha,\,\beta,\,\gamma)$ of partitions 
  such that $c_{\alpha\beta}^\gamma\neq 0$ and 
$l(\alpha)\leq e,\,l(\beta)\leq f,\,l(\gamma)\leq g$.
It is well known that $\LR(e,f,g)$ is a finitely generated semigroup
of $\NN^{e+f+g}$.\\

\subsubsection{Branching  semigroups}\label{sec:bsg}
Let $G$ be a connected reductive subgroup of a complex connected
reductive group $\hat{G}$.
Fix maximal tori $T\subset\hT$ and Borel subgroups $B\supset T$ and
$\hB\supset \hT$ of $G$ and $\hG$.
Let $X(T)$ denote the group of characters of $T$ and let $X(T)^+$
denote the set of dominant characters.
The irreducible representation of highest weight $\nu\in X(T)^+$ is denoted by $V_\nu$.
Similarly, we use the notation $X(\hat T)$,  $X(\hat T)^+$,  $V_{\hat\nu}$
relatively to $\hat G$.
The subspace of $G$-fixed vectors of the $G$-module $V$ is denoted by $V^G$.
Set
\begin{eqnarray}
 \label{eq:defc}
 c_{\nu\,\hat\nu}=\dim(V_\nu^*\otimes V_{\hat\nu})^G.
\end{eqnarray}
The branching problem is equivalent to the knowledge of these coefficients
since 
\begin{eqnarray}
  \label{eq:2}
  V_{\hat\nu}=\sum_{\nu\in X(T)^+}c_{\nu\,\hat\nu} V_\nu,
\end{eqnarray}
as a $G$-module.
Consider the set
$$
\LR(G,\hat G)=\{(\nu,\hat \nu)\in X(T)^+\times X(\hat T)^+\;:\;
c_{\nu\,\hat\nu}\neq 0\}.
$$
By a result of  Brion and Knop (see \cite{elash}),  $\LR(G,\hat G)$ is a finitely generated semigroup.

\subsubsection{GIT semigroups}\label{sec:gitsg}
Let $G$ be a complex reductive group acting on an irreducible
projective variety $X$.
Let $\Pic^G(X)$ denote the group of $G$-linearized line bundles on $X$.
The space $\Ho^0(X,\Li)$ of regular sections of $\Li$ is a $G$-module.
Consider the set
\begin{eqnarray}
 \label{eq:defLRX}
 \LR(G,X)=\{\Li\in\Pic^G(X)\;:\;
\Ho^0(X,\Li)^G\neq\{0\}\}.
\end{eqnarray}
Since $X$ is irreducible, the product of two nonzero $G$-invariant
sections is a nonzero $G$-invariant section and
$\LR(G,X)$ is a semigroup.

\subsection{Relations between these semigroups}

\subsubsection{Kronecker semigroups as branching semigroups.}
Let $E$ and $F$ be two complex vector spaces of dimension $e$ and $f$.
Consider the group $G=\GL(E)\times\GL(F)$.
Using Schur-Weyl duality, the Kronecker coefficient $g_{\alpha\,\beta\,\gamma}$ can be
interpreted in terms of representations of $G$. Namely  (see for
example \cite{Macdo,FH}) $g_{\alpha\,\beta\,\gamma}$ is the
multiplicity of $S^\alpha E\otimes S^\beta F$ in 
$S^\gamma(E\otimes F)$. 
More precisely, let $\gamma$ be a partition such that $l(\gamma)\leq
ef$. Then the simple $\GL(E\otimes F)$-module  $S^\gamma(E\otimes F)$ decomposes as a sum of simple $G$-modules as follows
\begin{eqnarray}
  \label{eq:KronGL}
  S^\gamma(E\otimes F)=\sum_{
    \begin{array}{c}
  \mbox{partitions }\alpha,\,\beta\mbox{ s.t.}\\
l(\alpha)\leq e,\,l(\beta)\leq f     
    \end{array}}
  g_{\alpha\,\beta\,\gamma}\;
S^\alpha E\otimes S^\beta F.
\end{eqnarray}
As a consequence
\begin{eqnarray}
  \label{eq:KronGL2}
  \Kron(e,f,ef)=\LR(\GL(E)\times \GL(F), \GL(E\otimes F))\cap (\ZZ^e\times  \ZZ^f\times(\NN)^{ef}). 
\end{eqnarray}

\subsubsection{Littlewood-Richardson semigroups as branching
  semigroups}
Since the Littlewood-Richardson coefficients are multiplicities for the tensor
product decomposition of $\GL_n$, we have
\begin{eqnarray}
  \label{eq:3}
\LR(e,e,e)=\LR(\GL_e, \GL_e\times \GL_e)\cap ((\NN)^e)^3.
\end{eqnarray}

\bigskip
The Littlewood-Richardson coefficients have another interpretation in
terms of representations of linear groups. Consider the embedding of
$\GL(E)\times\GL(F)$ in $\GL(E\oplus F)$ as a Levi subgroup by its
natural action on $E\oplus F$.
Then (see\cite[Chapter I, 5.9]{Macdo}) 
\begin{eqnarray}
  \label{eq:LRCoefLevi}
S^{\gamma}(E\oplus F)=\sum_{
    \begin{array}{c}
\mbox{partitions }\alpha,\,\beta\mbox{ s.t.}\\
l(\alpha)\leq e,\,l(\beta)\leq f     
    \end{array}} c_{\alpha\beta}^\gamma S^{\alpha}E\otimes S^{\beta}F.
\end{eqnarray}
In particular
\begin{eqnarray}
  \label{eq:34}
\LR(e,f,e+f)=\LR(\GL_e\times\GL_f, \GL_{e+f})\cap
(\ZZ^e\times\ZZ^f\times (\NN)^{e+f}).
\end{eqnarray}

\subsubsection{Branching semigroups as GIT
  semigroups}
  We use notation of Section~\ref{sec:bsg} and we assume that $G$ and
  $\hG$ are semisimple simply connected. 
  Consider the diagonal action of $G$ on $X=G/B\times \hG/\hB$. 
  Note that $\Pic^G(X)$ identifies with $X(T)\times X(\hT)$. 
  Then Borel-Weyl's theorem implies that 
  $\LR(G,\hG)=\LR(G,X)$.

  \subsubsection{Kronecker semigroups as GIT  semigroups}\label{sec:krongit}
  If $V$ is a complex finite dimensional vector space,  let  $\Fl(V)$ denote the variety of complete flags of $V$.
Given integers $a_i$ such that $1\leq a_1<\cdots<a_s\leq \dim (V)-1$,
we denote by   $\Fl(a_1,\cdots,a_s;\,V)$ the variety of flags
$V_1\subset\cdots\subset V_s\subset V$ such that $\dim(V_i)=a_i$ for
any $i$.
If $\alpha$ is a partition with at most dim$(V)$ parts then
$\Li_\alpha$ (resp. $\Li^\alpha$) denotes the $\GL(V)$-linearized line bundle on 
$\Fl(V)$ such that the space $\Ho^0(\Fl(V),\Li_\alpha)$ (resp.  $\Ho^0(\Fl(V),\Li^\alpha)$)
is isomorphic to $S^\alpha V^*$
(resp. $S^\alpha V$) as a $\GL(V )$-module.

Assume that  $E$ and $F$ are two linear spaces of dimension $e+1$ and $f+1$.
Set $G=\GL(E)\times \GL(F)$.  
Consider the variety
$$
X=\Fl(E)\times\Fl(F)\times\Fl(1,\cdots, e+f+1; E\otimes F)
$$
endowed with its natural $G$-action.
Let $\alpha$, $\beta$, and $\gamma$ be three partitions such that 
$l(\alpha)\leq e+1$,   $l(\beta)\leq f+1$, and   $l(\gamma)\leq e+f+1$.   
Consider the $\GL(E)$-linearized line bundle $\Li^\alpha$ on $\Fl(E)$,
 and respectively $\Li^\beta$ on $\Fl(F)$.
Since $l(\gamma)\leq e+f+1$, the line bundle $\Li_\gamma$ on
$\Fl(E\otimes F)$ is the pullback of a line bundle (still denoted by
$\Li_\gamma$) on 
$\Fl(1,\cdots, e+f+1; E\otimes F)$.
Consider the line bundle
$\Li=\Li^\alpha\otimes\Li^\beta\otimes\Li_\gamma$ on $X$ endowed with
its natural $G$-action.
Then 
$$
\Ho^0(X,\Li)\simeq S^\alpha E\otimes S^\beta F\otimes S^\gamma(E\otimes F)^*,
$$
and, by the formula~\eqref{eq:KronGL},
\begin{eqnarray}
  \label{eq:10}
  g_{\alpha\,\beta\,\gamma}=\dim(\Ho^0(X,\Li)^G).
\end{eqnarray}
The map $(\alpha,\beta,\gamma)\mapsto
\Li^\alpha\otimes\Li^\beta\otimes\Li_\gamma$ extends to a linear
isomorphism from
$\ZZ^{2e+2f+3}$ onto $\Pic^G(X)$.
This isomorphism allows to identify $\LR(G,X)$ with a
subset of $\ZZ^{2e+2f+3}$.
The equality~\eqref{eq:10} implies that
$$
\Kron(e+1,f+1,e+f+1)=(\ZZ_{\geq 0})^{2e+2f+3}\cap \LR(G,X).
$$

\section{Descriptions of branching and GIT cones}
\label{sec:GITEigen}  

  \subsection{GIT cones}
  
  Assume that the connected reductive group $G$ acts on the smooth projective
  variety $X$ and that $\Pic^G(X)$ has finite rank.
  Consider the cone $\QQ_{\geq 0}\LR(G,X)$ generated in $\Pic^G(X)\otimes\QQ$ 
  by the points of $\LR(G,X)$.
  The $G$-linearized ample line bundles on $X$ generated an open convex cone
  $\Pic^G(X)^+_\QQ$ in  $\Pic^G(X)\otimes\QQ$. 
  In this section, we recall from \cite{GITEigen} a description of the faces of  $\QQ_{\geq 0}\LR(G,X)$ that intersect $\Pic^G(X)^+_\QQ$.
  
  \bigskip Let $\Li$ be a $G$-linearized line bundle on $X$.
  Consider the associated set of  semistable points 
$$
X^{\rm ss}(\Li)=\{x\in X\;:\;\exists k>0\mbox{ and } \sigma\in
\Ho^0(X,\Li^{\otimes k})^G\qquad \sigma(x)\neq 0\}.
$$
Assume that $X^{\rm ss}(\Li)$ is nonempty. Then
  the projective variety $\Proj(\bigoplus_{k\geq 0}\Ho^0(X,\Li^{\otimes k})^G)$ is denoted by 
  $X^{\rm ss}(\Li)
  \quot G$.
  For later use, observe that $\dim(\Ho^0(X,\Li^{\otimes k})^G)$ is 
  $O(k^{\dim(X^{\rm ss}(\Li)  \quot G)})$.
If moreover $\Li$ is ample, 
$X^{\rm ss}(\Li)  \quot G$ is the categorical quotient of 
$X^{\rm ss}(\Li)$ by $G$.
In general, there is a  canonical $G$-invariant regular map
\begin{equation}
  \label{eq:defpi}
  \pi\,:\, X^{\rm ss}(\Li)\longto X^{\rm ss}(\Li)  \quot G.
\end{equation}

  Let $\lambda$ be a one parameter subgroup of $G$.
  The set $$
P(\lambda)=\{g\in G\,:\,\lim_{t\to 0}\lambda(t)g\lambda(t^{-1})\mbox{
 exists in } G\}
$$
is a parabolic subgroup of $G$. 
Consider an irreducible component $C$ of the fixed point set $X^\lambda$
of $\lambda$ in $X$. Set
$$
C^+=\{x\in X\,:\,\lim_{t\to 0}\lambda(t)x\in C\}.
$$
By  Bia\l{}ynicki-Birula's theorem, 
$C^+$ is an irreducible smooth locally closed subvariety of $X$.
Moreover it is stable by the action of $P(\lambda)$.
Consider on $G\times C^+$ the following action of the group
$P(\lambda)$:
$$
p.(g,x)=(gp^{-1},px).
$$
There exists a quotient variety denoted by 
$G\times_{P(\lambda)}C^+$. 
We denote by $[g:x]$, the class of $(g,x)\in G\times C^+$.
The following formula
$$
h.[g:x]=[hg:x]\qquad \forall h\in G,
$$
endows $G\times_{P(\lambda)}C^+$ with a $G$-action.
Consider the $G$-equivariant morphism
$$
\begin{array}{cccc}
  \eta\,:&G\times_{P(\lambda)}  C^+&\longto&X\\
&[g:x]&\longmapsto&gx.
\end{array}
$$
The pair $(C,\lambda)$ is said to be {\it well covering} if there
  exists a $P(\lambda)$-stable open subset $\Omega$ of $C^+$ such that 
  \begin{enumerate}
  \item the restriction of $\eta$ to $G\times_{P(\lambda)}\Omega$ is
    an open immersion;
\item $\Omega$ intersects $C$. 
  \end{enumerate}

\bigskip
For any $\Li\in \Pic^G(X)$, there exists an integer 
$\mu^\Li(C,\lambda)$  such that
$$
\lambda(t)\tilde z=t^{-\mu^\Li(C,\lambda)}\tilde z,
$$
for any $t\in\CC^*$, $z\in C$  and $\tilde z$ in the fiber $\Li_z$ over $z$ in $\Li$.

\begin{theo}(see \cite{GITEigen})
\label{th:gitcone}
\begin{enumerate}
\item For any well covering pair $(C,\lambda)$ and any $\Li\in \LR(G,X)$, we have
$\mu^\Li(C,\lambda)\leq 0$.
\item 
For any face $\Face$ of $\QQ_{\geq 0}\LR(G,X)$ intersecting $\Pic^G(X)^+_\QQ$ there exists a well covering pair $(C,\lambda)$ such that
$(\Li\otimes 1)\in\Face$ if and only if $\mu^\Li(C,\lambda)=0$, for
any ample $\Li$ in $\lr(G,X)$.
\item Let $(C,\lambda)$  be a well covering pair and $\Li$ be ample in $\LR(G,X)$. 
Then $\mu^\Li(C,\lambda)=0$ if and only if $X^{\rm ss}(\Li)\cap C$ is not empty.
\end{enumerate}
\end{theo}

\subsection{Branching cones}
With notation of Section~\ref{sec:bsg}, we want to describe the cone 
$\QQ_{\geq 0}\LR(G,\hG)$ generated by $\LR(G,\hG)$. 
We assume that  no nonzero  ideal of the Lie algebra $\Lie(G)$ of 
$G$ is an ideal of that $\Lie(\hG)$ of $\hG$:
this assumption implies that the cone $\QQ_{\geq 0}\LR(G,\hG)$ has nonempty interior in $(X(T)\times X(\hat T))\otimes\QQ$.

Consider the natural pairing  $\langle\cdot\,,\,\cdot\rangle$ between the one parameter subgroups and 
the characters of tori $T$ or $\hT$.
Let $W$ (resp. $\hW$) denote the Weyl group of $T$ (resp. $\hT$).
If $\lambda$ is a one parameter subgroup of $T$ (and thus of $\hT$), 
we denote by $W_\lambda$ (resp. $\hW_\lambda$) the stabilizer of $\lambda$ for the natural action
of the Weyl group.

The cohomology group ${\rm H}^*(G/P(\lambda),\ZZ)$ is freely generated by the Schubert classes
$\sigma_w$ parameterized by the  elements $w\in W/W_\lambda$.
Assume that $\lambda$ is dominant.
Let $w_0$ be the longest element of $W$. 
If $w\in W/W_\lambda$, we denote by $w^\vee\in W/W_\lambda$ the class of 
$w_0w$. By this way $\sigma_{w^\vee}$ and $\sigma_w$ are Poincar\'e dual. 
We  consider $\hG/\hP(\lambda)$, $\sigma_\hw$ as above but with $\hG$ in 
place of $G$.
Consider also the canonical $G$-equivariant immersion 
$\iota\,:\,G/P(\lambda)\longto\hG/\hP(\lambda)$; and the corresponding morphism  $\iota^*$
in cohomology. 

Recall from \cite{RR}, the definition of Levi-movability for  the pair 
$(\sigma_w,\sigma_\hw)$. For the purpose of this paper it is only useful to known that
if $(\sigma_w,\sigma_\hw)$ is Levi-movable then $\iota^*(\sigma_\hw).\sigma_w$ is a nonzero multiple of the class $[pt]$
of the point. Moreover the converse is true if  $\hG/\hP(\lambda)$
is minuscule.

Consider the set $\Wt_T(\hlg/\lg)$ of  nontrivial weights of $T$ in $\hlg/\lg$ and the set of  hyperplanes $H$ of $X(T)\otimes \QQ$ spanned by some
elements of $\Wt_T(\hlg/\lg)$.
For each such hyperplane $H$ there exist exactly two opposite indivisible one parameter subgroups 
$\pm\lambda_H$ which are orthogonal (for the paring $\langle\cdot,\cdot\rangle$) to $H$.
The so obtained one parameter subgroups are called
 {\it admissible} and   form a $W$-stable set.

\begin{theo}(see \cite{GITEigen})\label{th:gitEigen}

    Recall that no nonzero ideal of $\lg$ is an ideal of $\hlg$. 
Then, the cone $\lr(G,\hG)$ has nonempty interior  in $X(T\times\hT)\otimes \QQ$.

A dominant weight $(\nu,\hnu)$ belongs to  $\lr(G,\hG)$ if and only if
\begin{eqnarray}
  \label{eq:ineg}
\langle\hw\lambda,\hnu\rangle  \leq \langle w\lambda, \nu \rangle
\end{eqnarray}
for any dominant admissible one parameter subgroup $\lambda$ of $T$ 
and for any pair $(w,\hw)\in W/W_{\lambda}\times \hW/\hW_{\lambda}$ 
such that 
 
\begin{enumerate}
\item \label{cond:BeSj1}
$\iota^*(\sigma_\hw)\cdot\sigma_{w^\vee}=[pt]\in {\rm H}^*(G/P(\lambda),\ZZ)$, and 
\item  \label{cond:BeSj2}
the pair $(\sigma_{w^\vee},\sigma_\hw)$ is Levi-movable.
\end{enumerate}

Moreover, the inequalities~\eqref{eq:ineg} are pairwise distinct and no one can be omitted. 
\end{theo}

\section{Description of $\LR(e,f,e+f)$}
\label{sec:Horn}

\subsection{The statement}

\begin{theo}
\label{th:LR}
  Let $\alpha,\,\beta,$ and $\gamma$ be three partitions such that
  $l(\alpha)\leq e$, $l(\beta)\leq f$ and $l(\gamma)\leq e+f$.

Then $c_{\alpha\,\beta}^\gamma\neq 0$ if and only if
\begin{eqnarray}
  \label{eq:6}
  |\alpha|+|\beta|=|\gamma|,
\end{eqnarray}
and
\begin{eqnarray}
  \label{eq:93}
  \gamma_{f+i}\leq \alpha_i\leq\gamma_i,\quad \gamma_{e+j}\leq \beta_j\leq \gamma_j,
\end{eqnarray}
for any $i\in\{1,\dots,e\}$ and $j\in\{1,\dots,f\}$, and 
\begin{eqnarray}
  \label{eq:138}
 |\gamma_K|\leq   |\alpha_I|+|\beta_J|,
\end{eqnarray}
for any $0<r<e$ and $0<s<f$, for any 
$I\in\Part(r,e)$, $J\in\Part(s,f)$ and $K\in\Part(r+s,e+f)$ such that
\begin{eqnarray}
  \label{eq:7}
  c_{\tau^I\,\tau^J}^{\tau^K}=1.
\end{eqnarray}
Moreover,  the inequalities~\eqref{eq:93} or \eqref{eq:138} are pairwise distinct and no one can be omitted.  
\end{theo}

\bigskip
The partitions $\alpha$ and $\beta$ in the statement of Theorem~\ref{th:LR}  are also partitions of length at most $e+f$. Hence the nonvanishing of
$c_{\alpha\beta}^\gamma$ is equivalent to 
$(\alpha,\beta,\gamma)\in \lr(e+f,e+f,e+f)$.
But, by the classical Horn conjecture (see {\it e.g.} \cite{Fulton:survey}), this cone is
characterized by the inequalities
\begin{equation}
  \label{eq:81}
 |\gamma_{K'}|\geq |\alpha_{I'}|+|\beta_{J'}|
\end{equation}
where $\sharp I'=\sharp J'=\sharp K'$ and
\begin{equation}
  \label{eq:84}
  c_{\tau^{I'}\tau^{J'}}^{\tau^{K'}}=1.
\end{equation}
In some sense, Theorem~\ref{th:LR}  selects among the inequalities
\eqref{eq:81} those that remain essential when one imposes 
$l(\alpha)\leq e$ and $l(\beta)\leq f$.

Each inequality~\eqref{eq:138} has to be consequence of at least one
Horn inequality~\eqref{eq:81}. 
Indeed, by setting  $\tilde I=I_-\cup\{e+s+1,\dots,e+f\}$ and $\tilde
J=J_-\cup\{f+r+1,\dots,e+f\}$, 
one can check that, under the assumptions of Theorem~\ref{th:LR} and modulo the equality~\eqref{eq:6}, 
the inequality~\eqref{eq:138} is equivalent to 
\begin{eqnarray}
\label{eq:Horn}
|\gamma_{K_-}|\geq |\alpha_{\tilde I}|+|\beta_{\tilde J}|.
\end{eqnarray}
But $\sharp \tilde I=\sharp \tilde J=\sharp K_-=e+f-r-s$.
One can check that 
$$
c_{\tau^I\tau^J}^{\tau^K}=c_{\tau^{\tilde I}\tau^{\tilde J}}^{\tau^{K_-}}.
$$ 
Hence the assumption~\eqref{eq:7} implies that the
condition~\eqref{eq:Horn} is an Horn inequality \eqref{eq:81} for the
cone $\QQ_{\geq 0}\LR(e+f,e+f,e+f)$.

For the proof of Theorem~\ref{th:LR}, we need  to recall some notations and results on Schubert calculus on Grassmannians. 

\subsection{Schubert Calculus}
\label{sec:schubert}

Let $\Gr(r,n)$ be the Grassmann variety of $r$-dimensional linear subspaces
of  $V=\CC^n$.
Let $F_\bullet$: $\{0\}=F_0\subset F_1\subset 
F_2\subset\cdots\subset F_{n}=V$ be a complete flag   of $V$.
Let $I=\{i_1<\cdots<i_r\}\in\Part(r,n)$. 
The Schubert variety $X_I(F_\bullet)$ in $\Gr(r,n)$  is defined to be
$$
X_I(F_\bullet)=\{L\in\Gr(r,n)\,:\,
\dim(L\cap F_{i_j})\geq j {\rm\ for\ }1\leq j\leq r\}.
$$
The Poincar\'e dual of the homology class of $X_I(F_\bullet)$  is
denoted by $\sigma_I$.
It does not depend on $F_\bullet$.
The classes $\sigma_I$ form a $\ZZ$-basis of the cohomology ring of
$\Gr(r,n)$.
Recall from the introduction the definition of the partition $\tau^I$.
Then $\sigma_I$ has degree $2|\tau^I|$.
A first cohomological interpretation of the Littlewood-Richardson
coefficients is given by  the formula (see {\it e.g.} \cite{Manivel})
\begin{eqnarray}
  \label{eq:LRcohom1}
  \sigma_I.\sigma_J=\sum_{K\in\Part(r,n)}c_{\tau^I\,\tau^J}^{\tau^K} \sigma_K,
\end{eqnarray}
for any $I,\,J$ in $\Part(r,n)$.\\

Let $r$ and $s$  be two  integers such that $0<r<e$ and
$0<s<f$.
Fix an identification $\CC^{e+f}=\CC^e\oplus\CC^f$ and consider the
morphism
$$
\begin{array}{cccc}
  \phi_{r,s}\,:&\Gr(r,e)\times\Gr(s,f)&\longto&\Gr(r+s,e+f)\\
&(F,G)&\longmapsto&F\oplus G.
\end{array}
$$
The associated comorphism in cohomology is 
$$
\phi_{r,s}^*\,:\, \Ho^*(\Gr(r+s,e+f),\ZZ)\longto \Ho^*(\Gr(r,e)\times\Gr(s,f),\ZZ).
$$
By Kuneth's formula, the family 
$(\sigma_I\otimes\sigma_J)_{(I,J)\in\Part(r,e)\times\Part(s,f)}$ is a
basis of
$\Ho^*(\Gr(r,e)\times\Gr(s,f),\ZZ)$.
A second cohomological interpretation of the Littlewood-Richardson
coefficients is given by the formula 
\begin{eqnarray}
  \label{eq:LRcohom2}
\phi_{r,s}^*(\sigma_K)=\sum_{(I,J)\in\Part(r,e)\times\Part(s,f)}c_{\tau^I\,\tau^J}^{\tau^K} \;(\sigma_I\otimes\sigma_J),
\end{eqnarray}
for any $K\in\Part(r+s,e+f)$.

\subsection{Proof of Theorem~\ref{th:LR}}

  By the Knutson-Tao theorem of  saturation (see \cite{KT:saturation}),
  $c_{\alpha\,\beta}^\gamma\neq 0$ if and only if
  $(\alpha,\,\beta,\,\gamma)$ belongs to the cone $\QQ_{\geq
    0}\LR(e,f,e+f)$.
It remains to prove that the inequalities~\eqref{eq:93} and \eqref{eq:138}
characterize  the cone $\QQ_{\geq 0}\LR(e,f,e+f)$ in a minimal way.

\bigskip
Let us fix bases for the two vector spaces $E$ and $F$ of dimension $e$ and $f$.
Consider the group $\hat G=\SL(E\oplus F)$, its subgroup
$G=S(\GL(E)\times\GL(F))$ and on $E\oplus F$ the basis obtained by concatenating the bases
of $E$ and $F$.
Let $\hat T$ be the maximal torus of $\hat G$ consisting in diagonal 
matrices.
It is contained in $G$; set $T=\hT$.
Let $\hat B$ be the Borel subgroup of $\hat G$ consisting in upper
triangular matrices.
Set $B=\hat B\cap G$.
Let $\varepsilon_i$ be the character of $\hT$ mapping a matrix in $\hT$ to its 
$i^{\rm th}$ diagonal entry. Since $\sum_i\varepsilon_i=0$, $(\varepsilon_1,\dots,\varepsilon_{e+f-1})$ is a $\ZZ$-basis of $X(\hT)$.

Let $\alpha,\,\beta$, and $\gamma$ be three partitions of length less
or equal to $e$, $f$, and $e+f$.
The highest weight of the $\hG$-module $S^\gamma(E\oplus F)$ is
$\tilde\gamma=(\gamma_1-\gamma_{e+f})\varepsilon_1 +\cdots +
(\gamma_{e+f-1}-\gamma_{e+f})\varepsilon_{e+f-1} $.
The highest weight of the $G$-module $S^\alpha E\otimes S^\beta F$ is
$\widetilde{(\alpha,\beta)}=(\alpha_1-\beta_{f})\varepsilon_1 +\cdots +
(\alpha_e-\beta_{f})\varepsilon_{e}+ (\beta_1-\beta_{f})\varepsilon_{e+1} +\cdots +
(\beta_{f-1}-\beta_{f})\varepsilon_{e+f-1}$.
Then, by the formula~\eqref{eq:LRCoefLevi}
$$
\begin{array}{c@{\,}c@{\,}l}
(\alpha,\beta,\gamma)\in\LR(e,f,e+f)&\iff&
(S^\alpha E^*\otimes S^\beta F^*\otimes S^\gamma(E\oplus
F))^{\GL(E)\times \GL(F)}\neq 0,\\&\iff&
|\alpha|+|\beta|=|\gamma|
\\
&{\rm and}  &
(S^\alpha E^*\otimes S^\beta F^*\otimes S^\gamma(E\oplus
F))^G\neq 0,\\
&\iff&
|\alpha|+|\beta|=|\gamma|\\
&{\rm and}& 
(\widetilde{(\alpha,\beta)},\tilde \gamma)\in\LR(G,\hat G).
\end{array}
$$
In particular, to determine the inequalities for the cone $\QQ_{\geq 0}\LR(e,f,e+f)$, it is sufficient
to describe $\QQ_{\geq 0}\LR(G,\hG)$.
We do this using Theorem~\ref{th:gitEigen}.
The set of weights of $T$ acting on $\Lie(\hG)/\Lie(G)$ is  the 
set of weights of $T$ acting on $F^*\otimes E$ and their opposite.
Explicitly $\Wt_T(\Lie(\hG)/\Lie(G))=\pm\{\varepsilon_i-\varepsilon_{e+j}\,|\,
1\leq i\leq e\quad{\rm and}\quad
1\leq j\leq f\}$.
Let $(a_1,\dots,a_e,b_1,\dots,b_f)\in\ZZ^{e+f}$ be the exponents of a one parameter subgroup $\lambda$ of $T$; they satisfy $\sum_ia_i+\sum_j b_j=0$.
Then $\langle\lambda,\varepsilon_i-\varepsilon_{e+j}\rangle=0$
if and only if $a_i=b_j$. It follows that if $\lambda$ is admissible then the integers $a_i$ and $b_j$ 
take at most two values.
If moreover $\lambda$ is dominant then there exist integers $r,s$, and $c>d$ such that
$a_1=\cdots=a_r=b_1=\cdots=b_s=c$ and
$a_{r+1}=\cdots=a_e=b_{s+1}=\cdots=b_f=d$.
If moreover $\lambda$ is indivisible, $c=\frac{e+f-r-s}{(r+s)\wedge(e+f)}$ and 
$d=\frac{-r-s}{(r+s)\wedge(e+f)}$, where $\wedge$ denotes the gcd. 
Let $\lambda_{r,s}$ denote the so obtained one-parameter subgroup of $T$. 
Conversely, one easily checks that  $\lambda_{r,s}$ is an admissible dominant one-parameter subgroup of $T$, if $0<r<e$ and $0<s<f$ or if the pair $(r,s)$ is one of the four exceptional ones $\{
(1,0),(0,1),(e-1,f),(e,f-1)
\}$.

\bigskip
The inclusions $G/P(\lambda_{r,s})\subset \hG/\hP(\lambda_{r,s})$ associated to the four 
exceptional cases
are $\PP(E)\subset\PP(E\oplus F)$, $\PP(F)\subset\PP(E\oplus F)$,  $\PP(E^*)\subset\PP(E^*\oplus F^*)$ and $\PP(F^*)\subset\PP(E^*\oplus F^*)$. 
Consider $\PP(E)\subset\PP(E\oplus F)$. The restriction of $\sigma_{\{f+i\}}\in \Ho^*(\PP(E\oplus F),\ZZ)$ in  $\Ho^*(\PP(E),\ZZ)$ is  $\sigma_{\{i\}}$.
Then Theorem~\ref{th:gitEigen} implies that
$$
(e+f)\alpha_i-|\alpha|-|\beta|\geq (e+f)\gamma_{f+i}-|\gamma|.
$$
Modulo the identity~\eqref{eq:6}, this is equivalent to $\gamma_{f+i}\leq \alpha_i$.
Similarly, we get the three other inequalities~\eqref{eq:93}.

\bigskip
Fix now  $0<r<e$ and $0<s<f$.
The inclusion $G/P(\lambda_{r,s})\subset \hG/\hP(\lambda_{r,s})$ is the morphism $\phi_{r,s}$
defined in Section~\ref{sec:schubert}.
Consider  $\sigma_I\otimes\sigma_J\in \Ho^*(\Gr(r,e)\times\Gr(s,f),\ZZ)$
and $\sigma_K\in \Ho^*(\Gr(r+s,e+f),\ZZ)$ such that
$\phi_{r,s}^*(\sigma_K).(\sigma_I\otimes\sigma_J)^\vee=[pt]$.
Here the Levi movability is automatic since $\hG/\hP(\lambda_{r,s})$ is cominuscule. 
Modulo~\eqref{eq:6}, the inequality~\eqref{eq:ineg} of Theorem~\ref{th:gitEigen} corresponding to
$\sigma_I\otimes\sigma_J$ and $\sigma_K$ is the inequality~\eqref{eq:138}.
Then the theorem follows  from Theorem~\ref{th:gitEigen}.
\hfill $\square$

\subsection{Complement on stretched Littlewood-Richardson coefficients}

\begin{lemma}\label{lem:degLR}
 Let $\alpha,\,\beta,$ and $\gamma$ be three partitions such that
  $l(\alpha)\leq e$, $l(\beta)\leq f$ and $l(\gamma)\leq e+f$.

Then, the map  $n\longmapsto c_{n\alpha\,n\beta}^{n\gamma}$ is polynomial of degree not
greater than
$$
\binom{e} +\binom f +\binom{e+f}
-e^2-f^2+1,
$$
where $\binom{e} =\frac{e(e-1)}{2}$.
\end{lemma}

\begin{proof}
Since $c_{\alpha\beta}^\gamma=c_{\beta\alpha}^\gamma$, we may assume that $e\leq f$.
By \cite{DW:LRpol}, the function $\NN\longto\NN$, $n\longmapsto c_{n\alpha\,n\beta}^{n\gamma}$ is polynomial.

Recall that $E$ and $F$ are complex vector spaces of dimension $e$ and $f$.
Set $G=\GL(E)\times \GL(F)$ and $X=\Fl(E)\times\Fl(F)\times\Fl(E\oplus F)$.
Consider on $X$ the line bundle $\Li=\Li^\alpha\otimes\Li^\beta\otimes\Li_\gamma$.
Since $c_{n\alpha\,n\beta}^{n\gamma}=\dim(
\Ho^0(X,\Li^{\otimes n})^{G})$,
the degree of $c_{n\alpha\,n\beta}^{n\gamma}$ is equal to the dimension of 
$X^{\rm ss}(\Li)\quot G$.

Consider the map $\pi$ defined in \eqref{eq:defpi}. 
By Chevalley Theorem, since $\pi$ is dominant, for any general $y\in
X^{\rm ss}(\Li)$, one has
$$
\dim \pi\inv(\pi(y))=\dim(X^{\rm ss}(\Li))-\dim(X^{\rm ss}(\Li)\quot G).
$$
But, $\pi$ is $G$-invariant and $\pi\inv(\pi(y))$ contains $G.y$. Then
$$
\dim \pi\inv(\pi(y))\geq \dim(G.y)=\dim(G)-\dim(G_y),
$$
where $G_y$ is the stabilizer of $y$ in $G$.
But, for any $x\in X$, we have $\dim(G.x)\leq \dim(G.y)$ and
$$
\dim(X^{\rm ss}(\Li)\quot G)\leq \dim(X)-\dim(G)+\dim(G_x).
$$

We now claim that there exists $x$ such that $\dim(G_x)=1$.
Then the  lemma follows.

We now prove the claim by constructing explicitly $x$, that is, defining complete flags of $E$,
$F$ and $E\oplus F$. 
Fix bases $(\eta_1,\dots,\eta_e)$ and $(\zeta_1,\dots,\zeta_f)$ of $E$ and $F$.
On $E$ and $F$, we consider the two standard flags $F_\bullet^E$ and $F_\bullet^F$ in these bases.
Consider on $E\oplus F$, the following base
$$
(\eta_e+\zeta_f,\eta_e+\eta_{e-1}+\zeta_{f-1},\dots,
\eta_e+\cdots+\eta_{1}+\zeta_{f-e+1},\eta_1+\zeta_{f-e},\dots,\eta_1+\zeta_1)
$$
and the associated flag $F_\bullet^{E\oplus F}$.
One easily checks that $x=(F_\bullet^E,F_\bullet^F,F_\bullet^{E\oplus F})$ works.
\end{proof}

\section{Faces of $\QQ_{\geq 0}\Kron(e+1,f+1,e+f+1)$}

\subsection{Murnaghan's face}

The cone $\QQ_{\geq 0}\Kron(e+1,f+1,e+f+1)$ is contained in the
linear subspace  of points $(\alpha,\beta,\gamma)\in\QQ^{e+1}\times
\QQ^{f+1}\times \QQ^{e+f+1}$
that satisfy $|\alpha|=|\beta|=|\gamma|$.
In particular its dimension is at most $2e+2f+1$.

Recall that  $\bar\alpha=(\alpha_2\geq \alpha_3\cdots)$, 
if $\alpha=(\alpha_1\geq \alpha_2\cdots)$. 
 By Proposition~\ref{prop:Murna}, the points $(\alpha,\beta,\gamma)$ in $\QQ_{\geq
  0}\Kron(e+1,f+1,e+f+1)$ satisfy 
\begin{eqnarray}
  \label{eq:44}
  |\bar \alpha|+|\bar \beta|\geq |\bar\gamma|.
\end{eqnarray}
The set of points of  $\QQ_{\geq
  0}\Kron(e+1,f+1,e+f+1)$ such that equality holds in
the inequality~\eqref{eq:44} is a face $\Face^M$ ($M$ stands for Murnaghan) of the cone $\QQ_{\geq
  0}\Kron(e+1,f+1,e+f+1)$.
Consider the linear map
$$
\begin{array}{cccc}
  \pi\,:&\QQ^{2e+2f+3}&\longto&\QQ^{2e+2f}\\
&(\alpha,\beta,\gamma)&\longmapsto&(\bar\alpha,\bar\beta,\bar\gamma).
\end{array}
$$

\begin{lemma}
 \label{lem:FM}
The face $\Face^{M}$ 
maps by $\pi$ to $\QQ_{\geq 0}\LR(e,f,e+f)$.
Moreover  each fiber of $\pi$ over $\QQ_{\geq 0}\LR(e,f,e+f)$ contains
an unbounded interval.

The cone  $\QQ_{\geq
  0}\Kron(e+1,f+1,e+f+1)$  has dimension  $2e+2f+1$ and the face $\Face^{M}$ has dimension $2e+2f$. 
\end{lemma}

\begin{proof}
Assume that equality holds in the formula \eqref{eq:44}.
Assume also that the coordinates of $\alpha$, $\beta$, and $\gamma$
are nonnegative integers.
Then 
\begin{eqnarray}
  \label{eq:8}
  g_{\alpha\,\beta\,\gamma}=c_{\bar\alpha\,\bar\beta}^{\bar \gamma}.
\end{eqnarray}
Thus the face $\Face^{M}$ 
maps by $\pi$ on $\QQ_{\geq 0}\LR(e,f,e+f)$.
Conversely let $(\lambda,\mu,\nu)\in
\LR(e,f,e+f)$. Let $a$ be an integer and set   $b=a+|\lambda|-|\mu|$
and $c=a+|\lambda|-|\nu|$.
 If $a$ is big enough then 
$a\geq \lambda_1$, $b\geq\mu_1$ and  $c\geq \nu_1$.
Therefore $\alpha:=(a,\lambda)$, $\beta=(b,\mu)$ and $\gamma=(c,\nu)$
are three partitions of the same integer such that equality holds in the
inequality~\eqref{eq:44}.
Thus the equality~\eqref{eq:8} holds and $(\alpha,\beta,\gamma)$
belongs to $\Face^M$.
In particular the fiber $\pi^{-1}(\lambda,\mu,\nu)$ contains an
unbounded  segment.

Since $\QQ_{\geq 0}\LR(e,f,e+f)$ has dimension $2e+2f-1$ and the fibers
of $\pi$ have dimension at least one, the cone $\Face^M$ has dimension at least $2e+2f$.
 We had already noticed that  $\QQ_{\geq
  0}\Kron(e+1,f+1,e+f+1)$  has dimension at most $2e+2f+1$.
These two inequalities (and the fact that $\Face^M$ is a strict face of
the cone $\QQ_{\geq  0}\Kron(e+1,f+1,e+f+1)$ ) imply
the lemma.
\end{proof}

\subsection{Proof of Theorem~\ref{th:HornKron}}
\label{sec:proofmain}

Let $r$, $s$, $I$, $J$, and $K$ be like in Theorem~\ref{th:HornKron}.
To such a triple $(I,J,K)$, Theorem~\ref{th:LR} associates a codimension one face
of $\LR(e,f,e+f)$. Using Lemma~\ref{lem:FM}, this face corresponds to a face 
$\Face_{IJK}$ of $\QQ_{\geq 0}\Kron(e+1,f+1,e+f+1)$ of codimension two.
Explicitly, $\Face_{IJK}$ is the set of $(\alpha,\beta,\gamma)\in
\QQ_{\geq 0}\Kron(e+1,f+1,e+f+1)$ such that
\begin{eqnarray}
\left\{
\begin{array}{l}
|\bar\gamma|=
|\bar \alpha| +|\bar\beta|,\\
|\bar\gamma_K|=
|\bar \alpha_I| +|\bar\beta_J|.
\end{array}
\right .
\end{eqnarray}
This face $\Face_{IJK}$ is contained in two codimension one faces, $\Face^M$ and another one $\Face_{IJK}^M$ that we want to determine.

\bigskip
Let $\varphi_\tau$ denote the linear form defined by
$$
\varphi_\tau(\alpha,\beta,\gamma)=\tau(|\bar \alpha|+|\bar \beta|-
|\bar\gamma|)
+(|\bar \alpha_I|+|\bar \beta_J|- |\bar\gamma_K|),
$$
where $\tau$ is any rational number. Set also
$$
\varphi_\infty(\alpha,\beta,\gamma)=|\bar \alpha|+|\bar \beta|-
|\bar\gamma|.
$$
By the theory of convex polyhedral cones,  there exists $\tau_0$ such that for any $\tau>\tau_0$,
$\varphi_\tau$ is nonnegative on  the cone and the associated face is $\Face_{IJK}$, and, 
 $\varphi_{\tau_0}$ corresponds to $\Face^M_{IJK}$.

Here, $E$ and $F$ are two linear spaces of dimension $e+1$ and $f+1$ and $G=\GL(E)\times \GL(F)$.  
Consider the variety
$$
X=\Fl(E)\times\Fl(F)\times\Fl(1,\cdots, e+f+1; E\otimes F).
$$
We identify $\Pic^G(X)$ with $\ZZ^{2e+2f+3}$ like in Section~\ref{sec:krongit}.

\bigskip\noindent
{\bf Geometric description of $\varphi_\infty$.}
The inequality corresponding to $\Face^M$ is $\varphi_\infty\geq 0$. 
By Section~\ref{sec:krongit}, $\Face^M$ generates a face  of $\lr(G,X)$. 
Theorem~\ref{th:gitcone} shows that there exists a well covering pair $(C_\infty,\lambda_\infty)$ of $X$ such that $\varphi_\infty(\alpha,\beta,\gamma)=-\mu^{\Li^\alpha\otimes\Li^\beta\otimes\Li_\gamma}(C_\infty,\lambda_\infty)$.

To describe such a pair $(C_\infty,\lambda_\infty)$, fix  decompositions 
$E=\bar E\oplus l$ and $F=\bar F\oplus m$, where $\bar E$ and $\bar F$ are hyperplanes and $l$ and $m$ are lines. 
Let $\lambda_\infty$ be the one-parameter subgroup of $G$ acting 
with weight $1$ on $\bar E$ and $\bar F$, and with weight $0$ on $l$ and $m$.
Let $C_\infty$ be the set of points in $X$ such that
\begin{itemize}
\item the hyperplanes of the complete flags of $E$ and $F$ are
  respectively $\bar E$ and $\bar F$,
\item the line of the partial flag of $E\otimes F$ is $l\otimes m$,
\item the $(e+f+1)$-dimensional subspace of the partial flag of $E\otimes F$ is 
$(l\otimes m)\oplus (\bar E\otimes m)\oplus(l\otimes \bar F)$.
\end{itemize}
One can check that $(C_\infty,\lambda_\infty)$ works (see \cite{reduction} for details).

\bigskip\noindent
{\bf Geometric description of $\varphi_\tau$.}
Fix  decompositions 
$\bar E=E_+\oplus E_-$ and $\bar F=F_+\oplus F_-$, where $E_+$ and $F_+$ have dimension $r$ and $s$. 
Assume that $\tau>1$ and write $\tau=\frac p q$ with two integers  $p$ and $q$ satisfying $p\wedge q=1$ and $q>0$.
Let $\lambda_\tau$ be the  one parameter subgroup of $G$ acting  with weight $q+p$
on $E_+$ and $F_+$, with weight $p$ on $E_-$ and $F_-$
and with weight $0$ on $l$ and $m$.
The weight spaces of the action of $\lambda_\tau$ on
$E\otimes F$ are
$$
\begin{array}{|l|@{\,}c@{\,}|@{\,}c@{\,}|@{\,}c@{\,}|@{\,}c@{\,}|@{\,}c@{\,}|@{\,}c@{\,}|}
\hline
\mbox{Space}&
 E_+\otimes F_+&E_+\otimes F_-\oplus E_-\otimes F_+
& E_-\otimes F_-&
E_+\oplus F_+&
E_-\oplus  F_-&
l\otimes m\\
\hline
\mbox{Weight}&
2p+2q&2p+q&2p&p+q&p&0\\
\hline
 \end{array}
$$
where some ``$\otimes m$'' and ``$l\otimes$'' have been forgotten. 

To $I^\vee=\{e+1-i\,:\,i\in I\}$ is associated an embedding $\iota_{I^\vee}$ of $\Fl(E_+)\times \Fl(E_-)$
in $\Fl(\bar E)$. Explicitly
$$
\begin{array}{cccl}
\iota_{I^\vee}\,:&\Fl(E_+)\times \Fl(E_-)&\longto&\Fl(\bar E)\\
&((V_i),(W_j))&\longmapsto&(V_{\#I^\vee\cap [1,k]}\oplus W_{k-\#I^\vee\cap [1,k]})_{1\leq k\leq e.}
\end{array}
$$
Similarly we consider $\iota_{J^\vee}$ and $\iota_K$.
Observe that $C_\infty$ is canonically isomorphic to $\Fl(\bar E)\times \Fl(\bar F)\times \Fl(\bar E\oplus\bar F)$.
Consider the embedding $(\iota_{I^\vee},\iota_{J^\vee},\iota_K)$ of 
$$
 \Fl(E_+)\times \Fl(E_-)\times \Fl(F_+)\times \Fl(F_-)\times \Fl(E_+\oplus F_+)\times \Fl(E_-\oplus F_-)
 $$ in $C_\infty$.
Denote by $C_\tau$ its image. 
Using for example \cite[Proposition~1 and Theorem~1]{multi}, one can check
that, for $\tau$ big enough, $(C_\tau,\lambda_\tau)$ is a well covering pair. 
Moreover, 
$\varphi_\tau(\alpha,\beta,\gamma)=-q\mu^{\Li^\alpha\otimes\Li^\beta\otimes\Li_\gamma}
(C_\tau,\lambda_\tau)$.

For any $\tau>1$, $C_\tau$ is an irreducible component of
$\lambda_\tau$. Moreover, $C^+_\tau$ and $P(\lambda_\tau)$ do not depend on $\tau>1$. 
In Particular, $(C_\tau,\lambda_\tau)$ is a well covering pair for any $\tau>1$. 
 
 Theorem~\ref{th:gitcone} shows that the face determined by the inequality
 $\varphi_\tau$ only depends on $C_\tau$, and so does not depend on $\tau>1$: it is $\Face_{IJK}$.
 This implies that $\varphi_1\geq 0$ on $\QQ_{\geq 0}\Kron(e+1,f+1,e+f+1)$.
 Theorem~\ref{th:HornKron} follows.
 \hfill$\square$
 
 \bigskip
 \begin{NB}
 Let $\Face_1$ denote the face associated to $\varphi_1$. 
 Up to now, we have  not proved that $\Face_1$ has codimension 1 or equivalently that $\Face_1=\Face_{IJK}^M$. 
 This is the aim of Section~\ref{sec:codim}.
 \end{NB}
 
 \section{Proof of Theorem~\ref{th:reduction}}
 \label{sec:reduction}

Keeping the  notation of Section~\ref{sec:proofmain},
we  give a  geometric description of $\varphi_1$.
The weight spaces of the action of $\lambda_1$ on $E\otimes F$ are

$$
\begin{array}{|@{\,}c@{\,}|@{\,}c@{\,}|@{\,}c@{\,}|@{\,}c@{\,}|@{\,}c@{\,}|}
\hline
 E_+\otimes F_+&E_+\otimes F_-\oplus E_-\otimes F_+
& E_-\otimes F_-\oplus
E_+\oplus F_+&
E_-\oplus  F_-&
l\otimes m\\
\hline
4&3&2&1&0\\
\hline
 \end{array}
$$
The irreducible component $C_1$ of $X^{\lambda_1}$ containing $C_\tau$
(for $\tau>1$) is isomorphic to
$$
\begin{array}{ll}
& \Fl(E_+)\times \Fl(E_-)\times \Fl(F_+)\times \Fl(F_-)\\
 \times &\Fl(1,\dots,r+s;
 E_-\otimes F_-\oplus E_+\oplus F_+)\times \Fl(E_-\oplus F_-).
 \end{array}
$$
 Moreover, $C_1^+=C_\tau^+$ and $P(\lambda_1)=P(\lambda_\tau)$. In particular the pair $(C_1,\lambda_1)$ is well covering.\\
 
 Let $G^{\lambda_1}$ denote the centralizer of $\lambda_1$ in $G$.
Note that
$G^{\lambda_1}=\GL(E_+)\times \GL(E_-)\times \CC^*\times \GL(F_+)\times \GL(F_-)\times\CC^*$.
 By \cite[Theorem~2]{reduction}, $g_{\alpha\beta\gamma}$ is the dimension of
 $$
 \Ho^0(C_1,(\Li^\alpha\otimes\Li^\beta\otimes\Li_\gamma)_{|C_1})^{G^{\lambda_1}}.
 $$
We have to determine the restriction  $(\Li^\alpha\otimes\Li^\beta\otimes\Li_\gamma)_{|C_1}$ via the identification of $C_1$ with a product of flag
varieties. 
Fix a basis of $\bar E$ starting with a basis of $E_+$ followed by a basis of $E_-$.
For the group $\GL(\bar E)$ we consider standard maximal tori and Borel subgroups in this basis.
Similarly, we choose subgroups of $\GL(\bar F)$.
 
The maximal torus of $\GL(E)$ acts on the fiber in $\Li^\alpha$ over the 
base point of $\Fl(E)$ with weight  $(\alpha_{e+1},\dots,\alpha_1)$. 
The maximal torus $\bar T$ of $\GL(\bar E)$ acts on the fiber in $\Li^\alpha$ over the 
base point of $\Fl(\bar E)$ (embedded in $\Fl(E)$ like $C_\infty$ is embedded in $X$) 
with weight  $(\alpha_{e+1},\dots,\alpha_2)$. 
Let $w_{I^\vee}$ in the symmetric  group $S_e$ associated to
$I^{\vee}$ ($w_{I^\vee}(k)$ is the $k^{\rm th}$ elements of $I^\vee$
and $w_{I^\vee}(r+k)$ is the $k^{\rm th}$ elements of $I^\vee_-$).
Then  $\iota_{I^\vee}$ maps the base point of $\Fl(E_+)\times\Fl(E_-)$ to
 the image by $w_{I^\vee}^{-1}$
of the base point of $\Fl(\bar E)$. It follows that $\bar T$ acts on the fiber in $\iota_{I^\vee}^*(\Li^\alpha_{|C_\infty})$ by the weight $w_{I^\vee}^{-1}(\alpha_{e+1},\dots,\alpha_2)$.
After computation this gives 
$$
 \Ho^0(C_1,(\Li^\alpha\otimes\Li^0\otimes\Li_0)_{|C_1})=S^{\bar \alpha_{I_+}}E_+\otimes
S^{\bar \alpha_{I_-}}E_- .
 $$
 Similarly
 $$
 \Ho^0(C_1,(\Li^0\otimes\Li^\beta\otimes\Li_0)_{|C_1})=S^{\bar \beta_{J_+}}F_+\otimes
S^{\bar \beta_{J_-}}F_- ,
 $$
 and
 $$
 \Ho^0(C_1,(\Li^0\otimes\Li^0\otimes\Li_\gamma)_{|C_1})=S^{\bar \gamma_{K_+}}(E_-\otimes F_-\oplus E_+\oplus F_+)^*\otimes
S^{\bar \gamma_{K_-}}(E_-\oplus F_-)^*.
 $$
 We deduce that $g_{\alpha\beta\gamma}$ is the multiplicity of 
 the $\GL(E_+)\times \GL(E_-)\times \GL(F_+)\times \GL(F_-)$-simple module 
 $$
 S^{\bar \alpha_{I_+}}E_+\otimes
S^{\bar \alpha_{I_-}}E_-\otimes S^{\bar \beta_{J_+}}F_+\otimes
S^{\bar \beta_{J_-}}F_-
$$
 in the  module
\begin{eqnarray}
\label{eq:bigmod}
S^{\bar \gamma_{K_+}}(E_-\otimes F_-\oplus E_+\oplus F_+)\otimes
S^{\bar \gamma_{K_-}}(E_-\oplus F_-).
\end{eqnarray}

Now the theorem is obtained by using repeatedly the formulas~\eqref{eq:defLR}, \eqref{eq:KronGL} and
\eqref{eq:LRCoefLevi} to decompose the  module~\eqref{eq:bigmod}.

\section{Proof of Theorem~\ref{th:codimoneface}}
\label{sec:codim}

Recall that the aim is to prove that $\Face_1$ has codimension one. Since $\Face_{IJK}$ has codimension two and it is contained in $\Face_1$, it remains to prove that 
 $\Face_{IJK}\neq\Face_1$.\\

Assume now that   $(\alpha,\beta,\gamma)$ belongs $\Face_{IJK}$. Since $\Face_{IJK}$ is contained in $\Face^M$, 
 $g_{\alpha\beta\gamma}=c_{\bar \alpha\bar \beta}^{\bar \gamma}$. 
 Then (see {\it eg} \cite[Theorem 7.4]{DW:comb}) $g_{\alpha\beta\gamma}=c_{\bar \alpha_I\bar \beta_J}^{\bar \gamma_K}.c_{\bar \alpha_{I_-}\bar \beta_{J_-}}^{\bar \gamma_{K_-}}.$
 In particular, Lemma~\ref{lem:degLR} shows that   
 $g_{n\alpha\,n\beta\,n\gamma}$ is a polynomial function of $n$
 of degree at most
 \begin{equation}
\label{eq:dmax}
\begin{array}{l@{\,}l}
 d_{max}=&
\binom{r} +\binom s +\binom{r+s}+\binom{a} +\binom b +\binom{a+b}\\[1.1em]
&-r^2-s^2-a^2-b^2+2,
\end{array}
\end{equation}
 where $a=e-r$ and $b=f-s$.\\
 
 Given an algebraic group $\Gamma$ acting on an irreducible variety $Y$, we denote by
 $\mode(\Gamma,Y)$ the minimal codimension of the $\Gamma$-orbits.
 By \cite[Lemma~2]{algo}, for any $\Li$ in the relative interior of $\QQ_{\geq 0}\LR(G^{\lambda_1},C_1)$,
 the dimension of $C_1^{\rm ss}(\Li)\quot G^{\lambda_1}$ is equal to $\mode(G^{\lambda_1},C_1)$.

 By \cite[Theorem~4]{GITEigen}, there exists a line bundle $\Mi$ on $X$ such that $\Mi_{|C_1}$ belongs to the relative interior of $\QQ_{\geq 0}\LR(G^{\lambda_1},C_1)$. 
 We may assume that $\Mi=\Li^\alpha\otimes\Li^\beta\otimes\Li_\gamma$ for three partitions
 $\alpha$, $\beta$ and $\gamma$.
 But, by \cite[Theorem~8]{GITEigen}, $X^{\rm ss}(\Mi)\quot G\simeq C_1^{\rm ss}(\Mi_{|C_1})\quot G^{\lambda_1}$.
 It follows that $\Mi$ is a point on $\Face_1$ satisfying $\dim(X^{\rm ss}(\Mi)\quot G)=\mode(G^{\lambda_1},C_1)$.
  In particular, 
  
 \begin{equation}\label{eq:bigO}
   \begin{array}{ll}
 n\mapsto g_{n\alpha\,n\beta\,n\gamma}&\mbox{ cannot be a polynomial function}\\
 &\mbox{ of degree less than } \mode(G^{\lambda_1},C_1).
   \end{array}
 \end{equation}
 
 Regarding the assertions~\eqref{eq:dmax} and \eqref{eq:bigO}, to prove that $\Face_{IJK}\neq\Face_1$ it is sufficient to prove the claim: $\mode(G^{\lambda_1},C_1)>d_{max}$. 
The center of $G^{\lambda_1}$ contains a dimension 3 torus acting trivially on $C_1$.
Hence 
$$\mode(G^{\lambda_1},C_1)\geq \dim(C_1)-\dim(G^{\lambda_1})+3.$$
After simplification, we get 
$$
\mode(G^{\lambda_1},C_1)-d_{max}\geq ab(r + s)-1.
$$ 
 Since $a,b,r,$ and $s$ are positive integers, the claim follows.\\
 
 \section{Proof of Theorem~\ref{th:WeylKron}}
\label{sec:lestineq}

We keep notation of Section~\ref{sec:proofmain}, but now $r=e$ and 
$I=\{1,\dots,e\}$.
In particular $E_-$ is trivial and the weight spaces of the action of $\lambda_\tau$ of $E\otimes F$ are
$$
\renewcommand{\arraystretch}{1.2}
\begin{array}{|@{\,}c@{\,}|@{\,}c@{\,}|@{\,}c@{\,}|@{\,}c@{\,}|@{\,}c@{\,}|}
\hline
 \bar E\otimes F_+&\bar E \otimes F_-
&
\bar E\oplus F_+&
 F_-&
l\otimes m\\
\hline
2p+2q&2p+q&p+q&p&0\\
\hline
 \end{array}
$$

Hence $C_\tau\simeq\Fl(\bar E)\times\Fl(F_+)\times\Fl(\bar E\oplus F_+)$ for $\tau$ big enough.
Then, for any $\tau>0$, $(C_\tau,\lambda_\tau)$ is a well covering pair.  We conclude like in Section~\ref{sec:proofmain} that $\varphi_0$ is nonnegative on the Kronecker cone. This proves the first assertion of the theorem.\\

Consider now the limit case $\tau=0$.
The weight spaces are 
 $$
\renewcommand{\arraystretch}{1.2}
\begin{array}{|@{\,}c@{\,}|@{\,}c@{\,}|@{\,}c@{\,}|}
\hline
 \bar E\otimes F_+&
 \bar E \otimes (F_-
\oplus m)\oplus F_+
 &l\otimes(F_-
\oplus  m)\\
\hline
2&1&0\\
\hline
 \end{array}
$$
Hence $C_0\simeq\Fl(\bar E)\times\Fl(F_+)\times \PP(F_-\oplus  m)\times\Fl(1,\dots,e+f-1;\bar E\otimes (F_-\oplus m)\oplus F_+)\times \PP(F_-\oplus  m)$.
Moreover $G^{\lambda_0}\simeq \GL(\bar E)\times\CC^*\times \GL(F_+)\times \GL(F_-\oplus m)$.
By \cite[Theorem~2]{reduction}
$$
g_{\alpha\beta\gamma}=\dim(\Ho^0(C_0,(\Li^\alpha\otimes\Li^\beta\otimes\Li_\gamma)_{|C_0})^{G^{\lambda_0}}).
$$
The computation of this dimension is made using 
 the formulas~\eqref{eq:defLR}, \eqref{eq:KronGL} and
\eqref{eq:LRCoefLevi} like in Section~\ref{sec:reduction}.

\section{A final inequality}

All but two of the inequalities of Theorem~\ref{th:LR} had been extended to the Kronecker coefficients  by Theorems~\ref{th:WeylKron} and \ref{th:HornKron}.
The two exceptions are $\alpha_i\leq \gamma_i$ and $\beta_j\leq\gamma_j$.
Consider the second one, up to permuting $(\alpha,e)$ and $(\beta,f)$.
The extended inequality is
$$
\alpha_1+\beta_1-\beta_j-n\leq
\gamma_1-\gamma_j,
$$
for any $f+1\geq j\geq 2$.

This inequality is satisfied if $g_{\alpha\beta\gamma}\neq 0$. The proof is obtained by considering
$I=\emptyset$ and  $J=K=\{j-1\}$ in Section~\ref{sec:proofmain}.

\bibliographystyle{amsalpha}
\bibliography{hornkron}

\begin{center}
  -\hspace{1em}$\diamondsuit$\hspace{1em}-
\end{center}

\end{document}